\documentclass{amsproc}

\usepackage{overpic}[0]
\usepackage{mathtools}

\newtheorem{theorem}{Theorem}[section]
\newtheorem{lemma}[theorem]{Lemma}

\theoremstyle{definition}

\theoremstyle{remark}

\numberwithin{equation}{section}

\newcommand{\blankbox}[2]{%
  \parbox{\columnwidth}{\centering
    \setlength{\fboxsep}{0pt}%
    \fbox{\raisebox{0pt}[#2]{\hspace{#1}}}%
  }%
}

\usepackage{stackengine}
\usepackage{tikz-cd}
\newcommand\xrowht[2][0]{\addstackgap[.5\dimexpr#2\relax]{\vphantom{#1}}}

\begin{document}

\title{Counting cusp excursions of reciprocal geodesics}

\author{Ara Basmajian}
\address{PhD Program in Mathematics, The Graduate Center, CUNY, 365 Fifth  ave., NY,NY, 10016 and Department of Mathematics and Statistics, Hunter College, CUNY, 695 Park ave. NY, NY, 10065, USA}
\curraddr{}
\email{abasmajian@gc.cuny.edu}
\thanks{The first author was supported by a PSC-CUNY Grant and a grant from the Simons foundation (359956, A.B.)}

\author{Robert Suzzi Valli}
\address{Department of Mathematics, Manhattan College, 4513 Manhattan College Pkwy., Riverdale, NY, 10471}
\email{robert.suzzivalli@manhattan.edu}

\subjclass{Primary 20F69, 32G15, 57K20; Secondary 20H10, 53C22.}

\dedicatory{Dedicated to  Ravi S. Kulkarni on his 80th birthday}

\keywords{Reciprocal geodesic, modular group, modular surface, cusp excursion, word length}

\begin{abstract}
  For a fixed cusp neighborhood (determined by  depth $D$) of the modular surface, we  investigate the class of reciprocal geodesics  that  enter this   neighborhood (called a cusp excursion)  a fixed number of times. Since reciprocal geodesics traverse their  image an even number of times  such a geodesic only has an even number of cusp excursions,  say $2n$. Denoting this  class of geodesics as 
  $\mathcal{E}_{2n}^{D}$,  we  derive several growth rate results for 
  $|\{\gamma \in \mathcal{E}_{2n}^{D} : |\gamma|= 4t \}|$,
  where $|\gamma|$ is  the word length of $\gamma$ in 
 $\mathbb{Z}_2\ast \mathbb{Z}_3$ with respect to the generators of the  factors.

\end{abstract}

\maketitle
%
%

\section{Introduction}

For a fixed cusp neighborhood determined by its depth $D$,
the paper \cite{BasSuz} was interested in counting the number of so called 
$D$-low lying reciprocal geodesics on the modular surface. That is, in counting the  reciprocal geodesics that stay in the compact subsurface determined  by chopping off the $D$-cusp neighborhood.
In this paper, we  are interested in counting  the number of reciprocal geodesics on the modular surface
 that  leave this compact subsurface, called a cusp excursion,   a fixed number of times. Since reciprocal geodesics traverse their  image an even number of times  such a geodesic only has an even number of cusp excursions,  say $2n$. Denoting this  class of geodesics as 
  $\mathcal{E}_{2n}^{D}$,  we  derive  growth rate results 
  (Theorem  \ref{thm: main theorem}) for 
  $|\{\gamma \in \mathcal{E}_{2}^{D} : |\gamma|= 4t \}|$,
  and   $|\{\gamma \in \mathcal{E}_{2n}^{1} : |\gamma|= 4t \}|$.
Before we  state our theorem we first define our terms and set up  language.

 Consider the group  $G=\mathbb{Z}_2\ast \mathbb{Z}_3$ with 
 its natural symmetric generators   $a$, $b$, and $b^{-1}$ where 
 $a \in  \mathbb{Z}_2$ has order two, and $b \in \mathbb{Z}_3$ has order 3.  The representation of $G$ given by,
 $a \mapsto A$ and $b \mapsto B$ where,

$$
A=\begin{pmatrix} 
     0 & -1   \\
     1 & 0 
\end{pmatrix}, \text{ and } B=\begin{pmatrix} 
     1 & -1    \\
     1 & 0 
\end{pmatrix}
$$
 is a discrete faithful representation with image the modular group $PSL(2,\mathbb{Z}).$
The quotient  $S=\mathbb{H}/PSL(2,\mathbb{Z})$, called the {\it modular surface},  is a generalized pair of pants of  genus  zero and signature $(2,3,\infty)$.  For $\gamma$
a closed geodesic on $S$, its {\it word length}, denoted 
 $|\gamma|$, is the minimal word length 
 (with respect to the natural 
 generators of the factors in $\mathbb{Z}_2\ast \mathbb{Z}_3$)
 among the elements of the conjugacy class of a  lift
in  $PSL(2,\mathbb{Z})$.

Now, it is well-known that  if an orbifold surface has a cusp then it has a  cusp neighborhood, $\text{say }\mathcal{C}$, of area one  with  boundary a horocycle segment of length one.  For $D$ a positive integer, 
we define the {\it $D$-thick  part}    of $S$,
denoted $S_D$,  to be the smallest compact subsurface  of $S$ with horocycle boundary  which contains  all the  closed geodesics which when written in  normal form,  $ab^{\epsilon_1} ... ab^{\epsilon_t}$, 
($\epsilon_i=\pm 1 , \epsilon_1 \neq \epsilon_t$)
have   at most  $D$ consecutive $\epsilon_i 's$ of the same sign. The number of consecutive $\epsilon_i 's$ of the same sign (called a run)  correspond to how deep the geodesic goes into the cusp.   (See for example Lemma 7.1  in \cite{BasSuz}). Hence  it is natural to associate the geometric depth into the cusp with the size of a run; that is, the highest power of the parabolic. The $D$-thick parts form a compact exhaustion of $S$.

We  say that a closed geodesic (including reciprocal geodesics)
$\gamma$ has {\it $k$ (cusp) excursions  of depth greater than $D$}  if during one periodic orbit  $\gamma$ enters $S-S_D$  exactly $k$-times. See Figure \ref{fig1:cuspexcurs} for an example of a cusp excursion where $k=1$.

A {\it  reciprocal geodesic} on the modular surface is a  geodesic that begins and ends at the order two cone point, traversing its image twice if it is primitive, and more generally traversing the image $2m$-times  if it is the $m^{th}$ power of a reciprocal geodesic. Its lift is the  conjugacy class in $PSL(2,\mathbb{Z})$ of a hyperbolic element with axis passing through an order two fixed point. In the  paper \cite{BasSuz},  reciprocal geodesics that stay  in $S_D$ (so called $D$-low lying reciprocal geodesics) were studied.  In this paper, we are interested in the reciprocal geodesics that leave the $D$-thick part a prescribed (necessarily even) number of times (cusp excursions).
To that end,  for each  integer $n \geq 0$, define  

$$\mathcal{E}_{2n}^D=\{\gamma : \gamma \text{ a reciprocal geodesic that enters $S-S_D$  exactly   $2n$-times}\}
$$
Equivalently, $\gamma \in \mathcal{E}_{2n}^D$ if when represented in normal form  its
  associated run sequence has exactly $2n$ parts each  of size  at least  $D+1$. With this in mind, the reciprocal geodesics are partitioned into the disjoint sets

$${{\bigcup}_{n=0}^{\infty}}  \mathcal{E}_{2n}^D $$

\noindent The set  $\mathcal{E}_{0}^D$ is exactly the $D$-low lying reciprocals whose asymptotic growth rate is
$d_D \alpha_{D}^{t}$, where $d_D$ is given by the coefficient in  expression
(\ref{boundedcompositions}), 
and $\alpha_{D}$ is  the unique positive root of $z^D-z^{D-1}-\dots-1$. In stark contrast,   the set of all reciprocal geodesics has asymptotic growth  $2^{t-1}$ (see \cite{BasSuz}
for details).
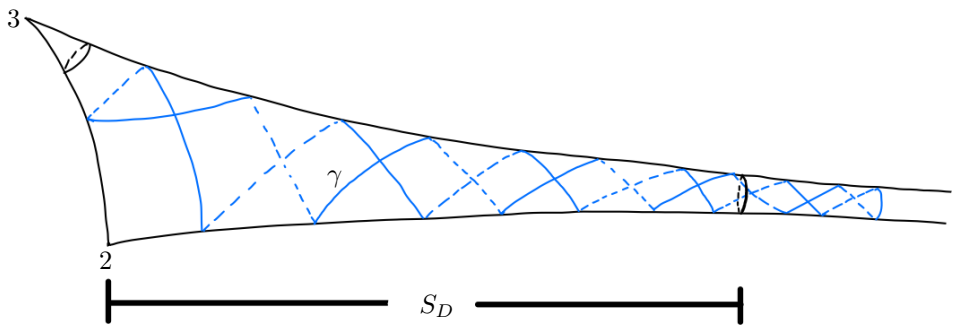
\begin{figure}

\begin{overpic}[scale=.45]{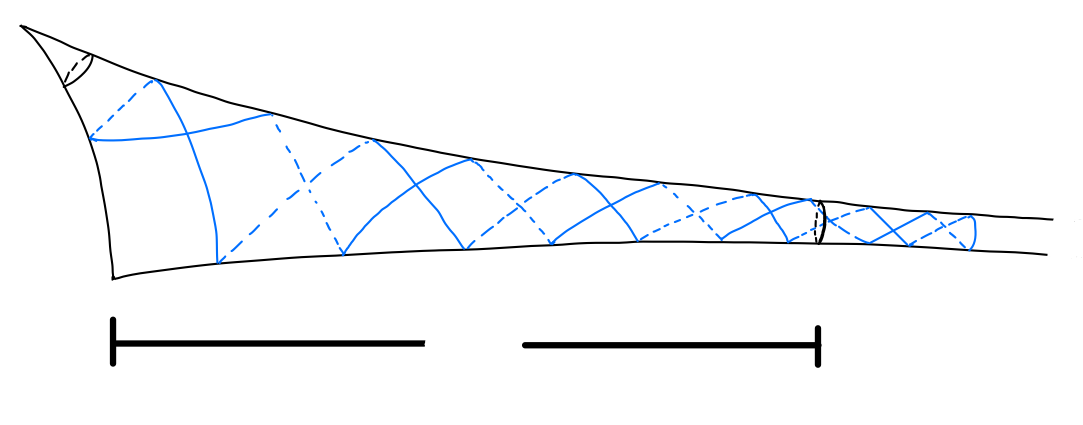}
\put(0,36){$3$}
\put(9.5,11){$2$}
\put(42.5,6.5){$S_D$}
\put(33,20){$\gamma$}
\end{overpic}
\caption{A cusp excursion of depth greater than $D$ on $S$}
\label{fig1:cuspexcurs}
\end{figure}
 So the first non-trivial excursions occur  in  $\mathcal{E}_{2}^{D}$. In this note we study the growth of reciprocal geodesics with a fixed number of excursions. Specifically, we are interested in how the word lengths of the geodesics in  
$\{\gamma \in \mathcal{E}_{2n}^{D} : |\gamma|= 4t \}$
grow  as a function of $t$.  
In this paper, we address the growth of reciprocal geodesics  with any number of excursions greater than depth $1$, and the growth of reciprocal geodesics with two excursions of depth greater than $D$.

\begin{theorem} For any integers $n\geq 0$,
 \label{thm: main theorem}
\begin{equation}
\frac{\big|\{\gamma \in  \mathcal{E}_{2n}^{1}: |\gamma|=4t\}\big|}
{t^{2n}} \xrightarrow{t\to\infty} \frac{1}{2n!} 
\end{equation}
\,\,\,\,\,\,\,\,\,\,\, \text{For  any integer $D\geq 2$},
\begin{equation} \label{eq: 2 excursions depth D}
 \frac{\big|\{\gamma \in  \mathcal{E}_2^D: |\gamma|=4t\}\big|}
{t  \alpha_D^{t}} \xrightarrow{t\to\infty} 
\frac{d_D^2}{\alpha_D^{D}(\alpha_D -1)}
\end{equation}
where $d_{D}=\frac{\alpha_D -1}{2+(D+1)(\alpha_D -2)}$,
and  $\alpha_D$ is the unique positive root of $z^D-z^{D-1}-\dots-1$.
\end{theorem}

In the sequel, Theorem  \ref{thm: main theorem} and its proof are  done in  two parts,  Theorems \ref{thm: 2n excursions} and 
\ref{thm: 2 excursions D-depth}.
For the set  
$\{\gamma \in  \mathcal{E}_{2n}^{1}: |\gamma|=4t\}$,  we in fact 
obtain a precise formula,  see 
Theorem  \ref{thm: 2n excursions}.
Moreover in the course of the proof of Theorem 
\ref{thm: 2 excursions D-depth}, we derive   coarse bounds (in closed form)  for the  set
 $\{\gamma \in  \mathcal{E}_{2}^{D}: |\gamma|=4t\}.$
As a result, in both cases,   it is straightforward  to  write down the growth rates of  $|\{\gamma \in  \mathcal{E}_{2n}^{1}: |\gamma|\leq 4t\}|$
  and  $|\{\gamma \in  \mathcal{E}_{2}^{D}: |\gamma|\leq 4t\}|$,
  as $t \rightarrow \infty$. Finally, we note that the  asymptotic constant in expression  (\ref{eq: 2 excursions depth D}) is only dependent on the depth $D$ and goes to zero  as $D \rightarrow \infty$. The more general setting of counting any number of excursions of depth bigger than $D$ will be addressed elsewhere. 
  
  The first author would like to thank the University of Michigan and the University of Luxembourg for their support while parts of this paper were being completed. In particular, thanks go to Richard D. Canary (University of Michigan) and 
  Hugo Parlier (University of Luxembourg)  for their hospitality during the respective visits.


\section{Some background and preliminaries}

There is an extensive history surrounding the modular group as well as an extensive history around  the question of geodesic growth in a hyperbolic surface. For the study of growth from a geometric point of view we refer to the papers, \cite{Boca,B-K2,B-K3,Bus,ErPaSo,ErSo,Mirz,Sar}.  From an algebraic point of view for the study of surface groups and free groups we refer to the papers,
\cite{CalLou,Er,GubSap,Mar,Park,Ri, Tra}. 
 Papers involving normal forms, enumeration schemes for curves, Farey arithmetic, and generating  elements in a non-abelian  free group can be found in \cite{BasLiu, Gil1,Gil2,GilKeen1,GilKeen2,GilKeen3,GilKeen4,Ser}.
 Papers specifically about reciprocal geodesics appear in  \cite{BasSuz, Boca, B-K3, ErSo1, Mar, Sar}.
 There is an extensive literature on cusp excursions by a random geodesic on a hyperbolic surface \cite{Haas1, Haas2, Haas3, Haas4, Haas5, MelPes, Mor, Poll, Ran, Strat, Sull, Trin}. In particular, the  papers    \cite{Sull, Haas1, Haas2, Haas3, Haas4, Haas5} investigate  the relation between depth, return time,  and other invariants in various contexts including connections to number theory.

  A closed geodesic in $S$ corresponds to a conjugacy class of hyperbolic elements in $G$. Using the fact that $ab$ and $ab^{-1}$ are parabolic elements,  a conjugacy class of hyperbolic elements  in  $PSL(2,\mathbb{Z})$ has a   representative of the  form  $ab^{\epsilon_1} ... ab^{\epsilon_t}$, where $\epsilon_i=\pm 1$  and $\epsilon_1 \neq \epsilon_t$. 
 Being  a product of the parabolic elements $ab$ and $ab^{-1}$, if $k$ consecutive $\epsilon_i$ have the same sign, then 
 $\gamma$ wraps around the cusp $k$-times, and hence the hyperbolic geometry dictates that  the geodesic $\gamma$ goes deeper into the cusp. 

  Now,  to a   $t$-tuple 
 $(\epsilon_1,...,\epsilon_t)$ we associate a {\it run}  sequence 
 ({\it composition} of $t$)
 $(r_1,...,r_k)$, where $r_1$ is the number of consecutive 
 $\epsilon_i$ of the same sign starting with $\epsilon_1$, $r_2$ is the length of the next such sequence, and so on. For example,  $(1,1,1,-1,1,-1,-1,1)$ corresponds to the composition
of 8,  $(3,1,1,2,1)$. Thus a composition of $t$, denoted
$C_{t}$,  is a sequence of positive integers that sum to $t$. The individual summands  of a composition  are  called its {\it parts}.
When the parts are at most $m$, we denote that subset of 
$C_{t}$ by $C_{t,m}$. Lastly, we denote a composition with exactly  $n$ parts that are bigger than $D$, and all other parts at most  $D$ by $C_{t}^{n,D}$.

 Geometrically the length of the run corresponds to how deep the geodesic goes into the cusp.   (See lemma 7.1  in \cite{BasSuz}).   Namely, for $D>0$, we say  a  closed geodesic  (reciprocal geodesics are included among the closed geodesics) is said to have a {\it cusp excursion}  (excursion for short) of depth bigger than $D$ if there is a run longer than $D$ in its normal form.  For reciprocal geodesics  note that the requirement that its orbit  be traversed an even number of times  is necessary to guarantee that the geodesic corresponds to an element of the modular group. Namely, a lift of such a geodesic corresponds to the axis of a hyperbolic element which passes through  an order two fixed point.  Since there is only one conjugacy class of order two fixed points  in $PSL(2,\mathbb{Z})$, this hyperbolic element can be conjugated so that its axis passes through $i$, the fixed point of $A$. As a consequence, a  word representing a reciprocal element (as an abstract group element) is conjugate to a word of the form  $a(\gamma a \gamma^{-1})$ and thus its translation length is realized on the surface as a geodesic that retraces its steps.  In terms of the generators of  $G=\mathbb{Z}_2\ast \mathbb{Z}_3$ it
has the {\it normal form},
$ab^{\epsilon_1} ... ab^{\epsilon_t} 
ab^{-\epsilon_t}...ab^{-\epsilon_1}$, where $\epsilon_i=\pm 1$.
Note  that the word length of a reciprocal geodesic is always a multiple of 4. The reciprocal word has a cusp excursion of depth 
bigger than $D$  if there are  at least (D+1) consecutive  
$\epsilon_i$ of  the same sign.  Table  \ref{table: main results}
has the association between various classes of reciprocal geodesics, the types of compositions they correspond to, and their growth rate.

\hspace{-1in}\begin{table}
\begin{tabular}{|l|l|l|l||||||||}

 \hline\xrowht{20pt}
{\bf Reciprocal Geodesics}  & \,{\bf Bijection To} & \,{\bf Cardinality}\\
\hline\xrowht{25pt}
    Length $4t$  & \, $ C_t$: compositions of $t$ & \,$=2^{t-1}$  \\  
    
    \hline\xrowht{25pt}
Length $4t$ in $S_D=\mathcal{E}_{0}^D$
 & \,\vtop{\hbox{\strut $C_{t,D}$:  Compositions of $t$}\hbox{\strut parts bounded by $D$}} & \,$\sim\left(\frac{\alpha_D -1}{2+(D+1)(\alpha_D -2)}\right)\alpha_{D}^{t} $    \\

\hline\xrowht{25pt}
   $\vtop{\hbox{\strut Length $4t$ in $\mathcal{E}_{2n}^{1}$}\hbox{\strut }}$  & $C_t^{2n,1}:$ $2n$ excursions of depth $>1$&  =${t}\choose{2n}$ \\

   \hline\xrowht{25pt}
  $\vtop{\hbox{\strut Length $4t$ in $\mathcal{E}_{2}^{D}$}\hbox{\strut  }}$ &  $C_t^{2,D}$: 
  $2$ excursions of depth $>D$ & 
  \vtop{\hbox{\strut  
  $\sim \left( \frac{d_{D}^{2}}{\alpha_{D}^{D}\left(\alpha_{D}-1\right)}\right) t\alpha_D^{t}$}\hbox{\strut }} \\
  
\hline

\end{tabular}
\vskip15pt
  \centering 
  \caption{Various classes of reciprocal geodesics and their growth}\label{table: main results}
\end{table}



 It is well-known that the primitive elements in the modular group $G$ are comprised of 
\begin{itemize}
\item one conjugacy class of order two corresponding to the element $a$,
\item one conjugacy class of order three corresponding to the element $b$ and its inverse,
 \item one conjugacy class of parabolic elements corresponding to the element $ab$ and its inverse, and
  \item infinitely many  conjugacy classes  of maximal cyclic  hyperbolic subgroups.
\end{itemize}

We remind the reader that a hyperbolic element in $PSL(2,\mathbb{Z})$
that represents a reciprocal geodesic on $S$ has axis passing through an order two fixed point. As a result, such a hyperbolic element is conjugate to itself in $PSL(2,\mathbb{Z})$.

Recall that   $C_{t,D}$ denotes  the set of compositions of $t$ with parts bounded by $D$. Counting this set  involves using the recursion relation: $|C_{t,D}|=\sum_{i=1}^D |C_{t-i,D}|$ for which details can be found in \cite{Dr}.    We have
\begin{equation}
\label{boundedcompositions}
|C_{t,D}|=
rnd\bigg(\frac{\alpha_D -1}{2+(D+1)(\alpha_D -2)}\alpha_D^{t}\bigg)
\end{equation}
where 
$rnd(x)=\lfloor{x+\frac{1}{2}}\rfloor$ and $\alpha_D$ is the unique positive root of the polynomial $z^D-z^{D-1}-\dots-1$ (see \cite{Dr} and \cite{BasSuz} for details).  
We remark that the coefficient $\frac{\alpha_D -1}{2+(D+1)(\alpha_D -2)}$ in equation (\ref{boundedcompositions}) only depends on $D$ and thus we denote it by $d_D$.  We note that  $2(1-2^{-D}) \leq  \alpha_D < 2$ and the $\alpha_D$ are increasing as $D$ increases. For the details see \cite{Dr}.


\section{Growth rates of reciprocal geodesics}

Let $n$ be some non-negative integer, and $D \geq 1$ a positive integer. The focus of this work is on reciprocal geodesics with $2n$ excursions of depth bigger than $D$. Namely,  we investigate  the  growth rates  of such    reciprocal geodesics and how they depend on the parameters $n$ and $D$. Since reciprocal geodesics always have an even number of excursions, we fix a positive integer $n$ and consider reciprocal geodesics with exactly $2n$ excursions of depth greater than $D$, and word length $4t$.  

We begin by setting up a bijection between reciprocal geodesics of length $4t$ on the modular surface, and compositions of $t$.
Let $\mathcal{N}_{4t}$ be the set of normal forms for a reciprocal word of length $4t$ in $G=\mathbb{Z}_2\ast\mathbb{Z}_3$, that is, elements of the form $ab^{\epsilon_1}\dots ab^{\epsilon_t}ab^{-\epsilon_t}\dots ab^{-\epsilon_1}, \epsilon_i=\pm1$.  Let $N_{4t}$ be the set of conjugacy classes of elements in $\mathcal{N}_{4t}$, i.e. the set of reciprocal geodesics with length $4t$. 
For each conjugacy class in $N_{4t}$ there are exactly two distinct representatives from $\mathcal{N}_{4t}$, see either Sarnak \cite{Sar} or Lemma 3.6 of \cite{BasSuz}.   Thus, the map $\xi:\mathcal{N}_{4t}\to N_{4t}$, which sends a normal form to its conjugacy class, is 2-1. 
Next, let $X_t=\{(\epsilon_1,\dots,\epsilon_t) : \epsilon_i=\pm1\}$, i.e. the set of binary $t$-tuples.  Then the map $\Phi: \mathcal{N}_{4t}\to X_t$, given by $\Phi(ab^{\epsilon_1}\dots ab^{\epsilon_t}ab^{-\epsilon_t}\dots ab^{-\epsilon_1})=(\epsilon_1,\dots,\epsilon_t)$, is 1-1. 
Now if we projectivize $X_t$, that is   $PX_t=X_t/\{\pm1\}$ and an  element of  $PX_t$ is the class, $\{(\epsilon_1,\dots,\epsilon_t), (-\epsilon_1,\dots,-\epsilon_t)\}$, we see that  the map $\tau:X_t\to PX_t$ which sends a $t$-tuple to its corresponding class, is 2-1. Lastly, the set of projectivized  binary $t$-tuples is in one to correspondence with the compositions of $t$. These correspondences are summarized in  figure \ref{fig: diagram}.
Thus it follows that $|N_{4t}|=\frac{1}{2}|\mathcal{N}_{4t}|=\frac{1}{2}|X_t|=\frac{1}{2}(2|PX_t|)=|PX_t|=|C_t|$, and we have shown that reciprocal geodesics are in 1-1 correspondence with compositions, $C_t$. In fact, more is true. 
This  1-1 correspondence persists  if we pass to the subclasses of reciprocal geodesics which are of interest in this article. Namely, the $D$-low lying reciprocal geodesics are in bijection with  compositions with parts bounded by $D$, and the 
reciprocal geodesics with $2n$ excursions of depth bigger than $D$ (that is, $\mathcal{E}_{2n}^D$) are in 1-1 correspondence with compositions with exactly $n$ parts that are bigger than $D$ (that is, $C_t^{n,D}$).




\begin{figure}\label{fig: diagram}
\begin{center}
\begin{tikzcd}
\mathcal{N}_{4t} \arrow[r, "1-1"] \arrow[d, "2-1"]
& X_t   \arrow[d, "2-1"] \\
N_{4t}
& PX_t \arrow[r, "1-1"]& C_{t}
\end{tikzcd}
\end{center}
\caption{Correspondence between $N_{4t}$ and   $C_{t}$}
\end{figure}

 \subsection{The case of   $2n$ excursions of depth bigger than 1}

Denote the subset of $N_{4t}$ consisting of reciprocal words with exactly $2n$ excursions  of depth greater than one by $N_{4t}^{(2n)}$. The corresponding subset of $PX_t$ consists of the classes of $t$-tuples in which there are exactly $n$ runs of $(+1)$ or $(-1)$ of length at least two starting from the left.  Denote this subset by $PX_t^{(n)}$.  (For example, the class $\{\pm(1,1,-1,1,-1,-1,-1)\}\in PX_7^{(2)}$, since there are seven entries with two runs of length at least two. The first is a run of two +1's and the second is a run of three -1's.)  Restricting the mappings from the diagram to these subsets, we have that $|N_{4t}^{(2n)}|=|PX_t^{(n)}|$.

 \begin{theorem} \label{thm: 2n excursions}
For  any integer $n \geq 0$,
\begin{displaymath}
\big|\{\gamma \in  \mathcal{E}_{2n}^{1}: |\gamma|=4t\}\big|={{t}\choose{2n}}   \text{ and } 
\frac{\big|\{\gamma \in  \mathcal{E}_{2n}^{1}: |\gamma|=4t\}\big|}
{t^{2n}} \xrightarrow{t\to\infty} \frac{1}{2n!}
\end{displaymath}

\end{theorem}

\begin{proof}
Fix integers $t>0, n\geq 0$.  We have reduced the problem to computing $|PX_t^{(n)}|$.  Let $x\in PX_t^{(n)}$.  Then $x$ has length $t$ with exactly $n$ runs of +1 or -1 of length at least two.  Each of the $n$ runs has a starting position and ending position which amounts to choosing $2n$ of the $t$ positions. 
 Thus, any such $x$ arises by selecting $2n$ slots among the $t$.  That is, $|PX_t^{(n)}|={{t}\choose{2n}}$.  It follows that $|PX_t^{(n)}|\sim \dfrac{t^{2n}}{(2n)!},$ as $t\to\infty$ ($t\in\mathbb{N}$).

\end{proof}

\subsection{The case of two excursions of depth  bigger than  $D$}

We next focus on reciprocal geodesics with two excursions. 

We, as usual, identify a  reciprocal geodesic of length $4t$ with its normal form in $PX_{t}^{1,D}$, the $t$-tuples with one run of length greater than $D$ and all other runs at most $D$. 
This projective binary sequence can in turn be identified with compositions of $t$ with one part, which we denote by $r$, greater than $D$,  and all other parts at most $D$. We   denote this set by $C_{t}^{1,D}$. In order to count this set we decompose it   in the following way. Let $r \geq D+1$ be the length of the run in $C_{t}^{1,D}$, and let $k$ be the position in the $t$-tuple $PX_{t}^{1,D}$ where the  $r$ run starts. Note that  $1 \leq k \leq t-r+1$. 
\vskip10pt

\begin{figure}
\vspace{.5in}
\begin{overpic}[scale=.75]{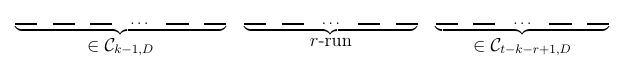}
\put(39.5,12){$\downarrow$}
\put(37,15){\text{$k^{th}$ slot}}
\put(70.5,12){$\downarrow$}
\put(67,15){$(k+r)^{th}$ slot}
\put(94.5,12){$\downarrow$}
\put(92,15){$t^{th}$ slot}
\end{overpic}
\caption{$2\leq k\leq t-r$}
\label{fig2:slotsarg}
\end{figure}
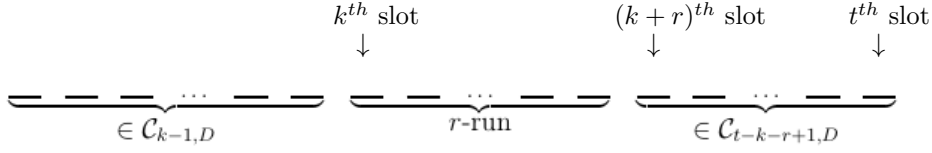
\vskip10pt


For $k$ and $r$ fixed and  in the permitted ranges listed above, we have

\begin{equation} \label{eq: counting excursions}
|C_{t}^{1,D}|=
|C_{k-1, D}||C_{t-k-r+1, D}|  \text{   for }  1 \leq k \leq t-r+1
\end{equation}
This follows  by counting the number of possible compositions to  the left and to the right of the $r$ run and then observing  that there is only one possible $r$ run starting in the $k^{th}$
slot.
See Figure \ref{fig2:slotsarg} for  a schematic. Here we use the convention that  $|C_{0, D}|=1$.

We remark that up until this point the discussion works for $D\geq 1$. For the rest of this section  we assume  $D \geq 2$.  
Summing over the admissible values of $k$ and $r$, we have that

\begin{equation} \label{eq: D excursion formula}
\big|\{\gamma \in  \mathcal{E}_2^D: |\gamma|=4t\}\big|
= \sum_{r=D+1}^{t} \sum_{k=1}^{t-r+1} |C_{k-1, D}||C_{t-k-r+1, D}| 
\end{equation}

Plugging  the appropriate values of expression  (\ref{boundedcompositions})  into  (\ref{eq: D excursion formula}) we obtain the lower estimate for 
$\big|\{\gamma \in  \mathcal{E}_2^D: |\gamma|=4t\}\big|$
as
    
 \begin{equation}
\sum_{r=D+1}^{t} \sum_{k=1}^{t-r+1} \left(d_{D} \alpha_{D}^{k-1}-\frac{1}{2} \right)
\left(d_{D} \alpha_{D}^{t-k-r+1}-\frac{1}{2} \right)
\end{equation}
and the upper estimate as,
\begin{equation} 
\sum_{r=D+1}^{t} \sum_{k=1}^{t-r+1}\left(d_{D} \alpha_{D}^{k-1}+\frac{1}{2} \right)
\left(d_{D} \alpha_{D}^{t-k-r+1}+\frac{1}{2} \right)
\end{equation}

The upper and lower estimates are of the form,

\begin{equation} \label{eq: lower and upper bounds}
\sum_{r} \sum_{k} d^{2} \alpha^{t-r} +
a\sum_{r} \sum_{k} d \alpha^{k-1}+
a\sum_{r} \sum_{k} d \alpha^{t-k-r+1}+
\sum_{r} \sum_{k} \frac{1}{4}
\end{equation}

where $a=\pm \frac{1}{2}$ and, for ease of notation, we have set $d=d_{D}$ and  $\alpha=\alpha_D$.

We are interested in what happens to the terms of  expression 
(\ref{eq: lower and upper bounds}) as $t \rightarrow \infty$.

To this end we first prove a lemma. 

\begin{lemma} \label{lem: growth of terms}
The sums in the following are over the admissible values of $r$ and $k$.

\begin{enumerate}
\item  $d^{2} \sum_{r} \sum_{k}  \alpha^{t-r} 
\sim \frac{d^{2}}{\alpha^{D} (\alpha-1)} t \alpha^{t}$
 \item $d \sum_{r} \sum_{k} \alpha^{k-1}=O(\alpha^{t})$
 \item $d \sum_{r} \sum_{k}  \alpha^{t-k-r+1}=O(\alpha^{t})$

\end{enumerate}

\end{lemma}

\begin{proof} To ease notation, we use the convention 
$d=d_{D}$ and $\alpha=\alpha_{D}$.
To prove item (1), 
note that $d^{2} \sum_{k}  \alpha^{t-r}
=d^{2} \alpha^{t-r} \left(t-r+1 \right)$. 
Summing this over the admissible $r$'s we have,

$$
d^{2} \sum_{r=D+1}^{t} \alpha^{t-r} \left(t-r+1 \right)=
$$
$$
d^{2}\alpha^{t} \left[ t \sum_{r=D+1}^{t}  \alpha^{-r}  
-   \sum_{r=D+1}^{t} (r-1)\alpha^{-r} \right]=
$$
$$
\left(d^{2} \sum_{r=D+1}^{t} \alpha^{-r}\right)t\alpha^{t}
- \left( d^{2}  \sum_{r=D+1}^{t} (r-1)\left(\frac{1}{\alpha}\right)^{r}
\right)\alpha^{t}=
$$
$$
d^{2} \left(\frac{\alpha^{-D}-\alpha^{-t}}{\alpha-1}\right) t\alpha^{t} +O(\alpha^{t}).
$$

To prove item (2),  we first note that a short computation yields

$$
\alpha^{-1} \sum_{k=1}^{t-r+1} \alpha^{k}=
\frac{\alpha^{t-r+1}-1}{\alpha -1}.
$$
Now, summing over the admissible $r$'s we obtain,
$$
\sum_{r=D+1}^{t} \frac{\alpha^{t-r+1}-1}{\alpha -1}=
$$
$$
\left( \frac{1}{\alpha-1}\right) 
\left[ \alpha^{t+1} \sum_{r=D+1}^{t} \alpha^{-r} 
-(t-D) \right]=O(\alpha^{t}).
$$

The proof of item (3) follows in a similar fashion  to item (2).
\end{proof}

 We are now ready to prove

\begin{theorem} \label{thm: 2 excursions D-depth}
For $D \geq 2$, 
\begin{displaymath}
\frac{\big|\{\gamma \in  \mathcal{E}_2^D: |\gamma|=4t\}\big|}
{t  \alpha_D^{t}} \xrightarrow{t\to\infty}
\frac{d_D^2}{\alpha_D^{D}(\alpha_D -1)}. 
\end{displaymath}
\end{theorem}

\begin{proof}
Consider the terms of expression 
(\ref{eq: lower and upper bounds}).
Lemma  \ref{lem: growth of terms} and the fact that the fourth term 
 grows like $t^2$, ensures that  the  first term, whose asymptotic is  $t \alpha^{t}$,  dominates the expression as $t \rightarrow \infty$.  Since the first term does not depend on the constant $a$, we may conclude that the lower and upper estimates of 
${\big|\{\gamma \in  \mathcal{E}_2^D: |\gamma|=4t\}\big|}$   have the same asymptotic.
\end{proof}

\bibliographystyle{amsalpha}

\begin{thebibliography}{A}

\bibitem{BasLiu} A.~Basmajian and M.~Liu, Low-lying geodesics on the modular surface and necklaces. \emph{Arxiv}

\bibitem{BasSuz} A.~Basmajian and R.~Suzzi Valli, Combinatorial growth in the modular group. \emph{Groups Geom. Dyn.} \textbf{16} (2022), no. 2, 683--703

\bibitem{Boca} F.~Boca, V.~Pa\c sol, A.~Popa, and A.~Zaharescu, Pair correlation of angles between reciprocal geodesics on the modular surface. 
\emph{Algebra Number Theory} \textbf{8} (2014), no.~4, 999--1035

\bibitem{B-K2} J.~Bourgain and A.~Kontorovich, Beyond
expansion II: low-lying fundamental geodesics. \emph{J. Eur. Math. Soc. (JEMS)} \textbf{19} (2017), no.~5, 1331--1359


\bibitem{B-K3} J.~Bourgain and A.~Kontorovich, Beyond
expansion III: Reciprocal geodesics. \emph{Duke Math J.} \textbf{168} (2019), no.~18, 3413--3435

\bibitem{Bus} P. Buser, {\it Geometry and spectra of compact Riemann surfaces.} Progress in Mathematics, 106. Birkh\"{a}user Boston,  Inc., Boston, MA 1992. xiv+454 pp.


\bibitem{CalLou} D.~Calegari and J.~Louwsma, Immersed surfaces in the modular orbifold. \emph{Proc. Amer. Math. Soc.} \textbf{139} (2011), no.~7, 2295--2308



\bibitem{Dr} G.~Dresden and Z.~Du, A simplified Binet formula for $k$-generalized Fibonacci numbers. \emph{J. Integer Seq.} \textbf{17} (2014), no.~4, Article 14.4.7, 9 pp.

\bibitem{Er} V.~Erlandsson, A remark on the word length in surface groups. \emph{Trans. Amer. Math. Soc.} \textbf{372} (2019), no.~1, 441--445


\bibitem{ErPaSo} V.~Erlandsson, H.~Parlier, and J.~Souto, Counting curves, and the stable length of currents. \emph{J. Eur. Math. Soc. (JEMS)} \textbf{22} (2020), no.~6, 1675--1702


\bibitem{ErSo1} V.~Erlandsson and J.~Souto, Counting and equidistribution of reciprocal geodesics and dihedral groups.\texttt{ arXiv:2204.09956}

\bibitem{ErSo} V.~Erlandsson and J.~Souto, Counting curves in hyperbolic surfaces. \emph{Geom. Funct. Anal.} \textbf{26} (2016), no~3, 729--777





\bibitem{Gil1} J.~Gilman, Informative words and discreteness. \emph{Combinatorial group theory, discrete groups, and number theory}, 
 147--155, Contemp. Math., 421, Amer. Math. Soc., Providence, RI,  2006


\bibitem{Gil2} J.~Gilman, Primitive curve lengths on pairs of pants. \emph{Infinite group theory}, 
 141--155, World Sci. Publ., Hackensack, NJ,  2018


\bibitem{GilKeen1} J.~Gilman and L.~Keen, Cutting sequences and palindromes. \emph{Geometry of Riemann surfaces}, 
 194--216, London Math. Soc. Lecture Note Ser., 368, Cambridge Univ. Press, Cambridge,  2010


\bibitem{GilKeen2} J.~Gilman and L.~Keen, Enumerating palindromes and primitives in rank two free groups.
\emph{ J. Algebra}  \textbf{332}  (2011), 1--13


\bibitem{GilKeen3} J.~Gilman and L.~Keen, Lifting free subgroups of $PSL(2,\mathbb{R})$ to free groups. \emph{Quasiconformal mappings, Riemann surfaces, and Teichmüller spaces}, 
 109--122, Contemp. Math., 575, Amer. Math. Soc., Providence, RI,  2012


\bibitem{GilKeen4} J.~Gilman and  L.~Keen, Word sequences and intersection numbers.
\emph{Complex manifolds and hyperbolic geometry} (Guanajuato, 2001), 231--249, Contemp. Math., 311, Amer. Math. Soc., Providence, RI,  2002


\bibitem{GubSap} V.~Guba and M.~Sapir, On the conjugacy growth functions of groups. \emph{Illinois J. Math.} \textbf{54} (2010), no. 1, 301--313



\bibitem{Haas1} A. Haas, Excursion and return times of a geodesic to a subset of a hyperbolic Riemann surface.
 \textit{Proc. Amer. Math. Soc.}  \textbf{141}  (2013),  no. 11, 3957--3967
		
\bibitem{Haas2} A. Haas, Geodesic cusp excursions and metric Diophantine approximation.
 \textit{Math. Res. Lett.}  \textbf{16}  (2009),  no. 1, 67--85
		
\bibitem{Haas3} A. Haas, Geodesic excursions into an embedded disc on a hyperbolic Riemann surface.
\textit{Conform. Geom. Dyn.}  \textbf{13}  (2009), 1--5
		
\bibitem{Haas4} A. Haas, The distribution of geodesic excursions into the neighborhood of a cone singularity on a hyperbolic 2-orbifold.
 \textit{Comment. Math. Helv.}  \textbf{83}  (2008),  no. 1, 1--20
		
\bibitem{Haas5} A. Haas, The distribution of geodesic excursions out the end of a hyperbolic orbifold and approximation with respect to a Fuchsian group.
\textit{Geom. Dedicata}  \textbf{116}  (2005), 129--155







 \bibitem{Mar} B.~Marmolejo, {\it Growth of Conjugacy Classes of Reciprocal Words in Triangle Groups.}
Ph.D. thesis, City University of New York, 2020


\bibitem{MelPes} M. V. Melián and D. Pestana, Geodesic excursions into cusps in finite-volume hyperbolic manifolds.
 \textit{Michigan Math. J.}  \textbf{40}  (1993),  no. 1, 77--93

\bibitem{Mirz} M.~Mirzakhani, Growth of the number of simple closed geodesics on hyperbolic surfaces. \emph{Ann. of Math. (2)} \textbf{168} (2008), no.~1, 97--125

\bibitem{Mor} R. Mor, Excursions to the cusps for geometrically finite hyperbolic orbifolds and equidistribution of closed geodesics in regular covers.
 \textit{Ergodic Theory Dynam. Systems}  \textbf{42}  (2022),  no. 12, 3745--3791

\bibitem{Park} P.~Park, Conjugacy growth of commutators. \emph{J. Algebra} \textbf{526} (2019), 423--458

\bibitem{Poll} M. Pollicott, Limiting distributions for geodesics excursions on the modular surface.
 Spectral analysis in geometry and number theory, 
 177--185, Contemp. Math., 484, Amer. Math. Soc., Providence, RI,  2009

\bibitem{Ran} A. Randecker and G. Tiozzo, Cusp excursion in hyperbolic manifolds and singularity of harmonic measure.
\textit{ J. Mod. Dyn.}  \textbf{17}  (2021), 183--211

 \bibitem{Ri} I.~Rivin, Growth in free groups (and other stories) -- twelve years later. \emph{Illinois J. Math.} \textbf{54} (2010), no.~1, 327--370
 
\bibitem{Sar} P.~Sarnak, Reciprocal geodesics. \emph{Analytic number theory}, 217--237, Clay Math. Proc., 7, Amer. Math. Soc., Providence, RI, 2007
\MR{2362203}

\bibitem{Ser} C. Series, The modular surface and continued fractions. \textit{J. London Math. Soc.} (2)  \textbf{31}  (1985),  no. 1, 69--80

\bibitem{Strat} B. Stratmann, A note on counting cuspidal excursions.
 \textit{Ann. Acad. Sci. Fenn. Ser. A I Math.}  \textbf{20}  (1995),  no. 2, 359--372

\bibitem{Sull} D. Sullivan, Disjoint spheres, approximation by imaginary quadratic numbers, and the logarithm law for geodesics.
 \textit{Acta Math.}  \textbf{149}  (1982),  no. 3-4, 215--237


\bibitem{Tra} C. Traina, {\it The conjugacy problem of the modular group and the class number of real quadratic number fields.} J. Numer Theory 21 (1985), no. 2, 176-184.






	
		
\bibitem{Trin} M. Trin, Thurston's compactification via geodesic currents: The case of non-compact finite area surfaces. \texttt{	arXiv:2208.10763}








\end{thebibliography}

\end{document}